\documentclass[a4paper,12pt]{amsart}

\usepackage[colorlinks,citecolor = red, linkcolor=blue,hyperindex]{hyperref}
\usepackage{euscript,eufrak,verbatim}
\usepackage[psamsfonts]{amssymb}
\usepackage[usenames]{color}
\usepackage[curve]{xypic}
\usepackage{graphicx}
\usepackage{mathrsfs}

\usepackage{float}
\usepackage{bbm}
 \usepackage{color}
\newtheorem{thm}{Theorem}[section]
\newtheorem{lem}[thm]{Lemma}

\newtheorem{prop}[thm]{Proposition}
\newtheorem{cor}[thm]{Corollary}

\newtheorem{rem}{Remark}[section]

\newtheorem*{prob*}{Problem}

\def\R{\mathbb{R}}
\def\N{\mathbb{N}}
\def\Z{\mathbb{Z}}

\def\X{\mathbb{X}}

\def\1{\mathbbm{1}}
\def\an {\text{\, and \,}}

\def\supp{\mathrm{supp}}

\def\K{\mathcal{K}}
\def\X{\mathcal{X}}
\def\clos{\mathrm{clos}}

\def\XXint#1#2#3{{\setbox0=\hbox{$#1{#2#3}{\int}$ }
\vcenter{\hbox{$#2#3$ }}\kern-.6\wd0}}

\begin{document}

\title[Isometric actions of inductively compact groups]{Ergodic measures on compact metric spaces for isometric actions  by inductively compact groups}


\author
{Yanqi Qiu}
\address
{Yanqi QIU: CNRS, Institut de Math{\'e}matiques de Toulouse, Universit{\'e} Paul Sabatier, 118 Route de Narbonne, F-31062 Toulouse Cedex 9, France}

\email{yqi.qiu@gmail.com}

\thanks{This work is supported by the grant IDEX UNITI - ANR-11-IDEX-0002-02, financed by Programme ``Investissements d'Avenir'' of the Government of the French Republic managed by the French National Research Agency.}

\begin{abstract}
 We obtain  a partial converse of Vershik's description of ergodic probability measures  on a compact metric space with respect to an isometric action by an inductively compact group. This allows us to identify, in this setting,  the set of ergodic probability measures  with the set of weak limit points of orbital measures. 
\end{abstract}


\maketitle

\section{Introduction}

A topological group is called inductively compact if it is the inductive limit of an increasing chain of compact groups.  In this note, we will investigate the ergodic probability measures of on a compact metric space with respect to an isometric action by an inductively compact group.  

Let $\K(1) \subset \cdots \subset \K(n)\subset \cdots$ be an increasing chain of compact groups, denote the corresponding inductive limit by $\K(\infty)$. Let $\mathcal{X}$ be a separable metric complete space. Assume that $\K(\infty)$ acts on $\X$ by Borel actions. By a result of Vershik, all ergodic probability measures on $\X$ are the weak limits of orbital measures. Let us state this more precisely. Denote by $m_n$ the normalized Haar measure on $\K(n)$. For any point $x\in \X$, we denote by $\mu_n^x$ the unique $\K(n)$-invariant probability measure on the orbit 
$$\K(n)\cdot x  =  \{y \in \X: y = g \cdot x \text{\, for some $g \in\K(n)$}\}.$$
Let us denote by $\mathcal{P}(\X)$ the set of all Borel probability measures on $\X$. 
Recall that by definition, a sequence  $(\nu_n)_{n\in\N}$ of Borel probability measures on $\X$ converges weakly to a Borel probability measure $\nu \in \mathcal{P}(\X)$ if  for any bounded continuous function $f$, we have 
$$
 \int_{\X} f d\nu = \lim_{n\to\infty}  \int_{\X} f d\nu_n.
$$
In this situation, we denote $\nu_n \Longrightarrow \nu$ as $n \to \infty$.

\begin{thm}[{Vershik \cite[Theorem 1]{Vershik-inf-group}}]\label{Vershik}
Let $\mu$ be an ergodic $\mathcal{K}(\infty)$-invariant Borel probability measure on $\mathcal{X}$. Then for $\mu$-almost every point $x\in\mathcal{X}$, we have 
$$
\text{$\mu_n^x \Longrightarrow \mu$ as $n \to \infty$.} 
$$
\end{thm}

The purpose of this note is to give a partial converse of Theorem \ref{Vershik}. 

\begin{thm}\label{main-thm}
Let $\mathcal{X}$ be a {\it compact metric space}. Assume that $\K(\infty)$ acts on $\X$ by isometric isomorphisms.  If $\mu \in\mathcal{P}(\X)$ is a weak limit point of the sequence of orbital measures $(\mu_n^{x_0})_{n\in\N}$, then $\mu$ is ergodic.  
\end{thm}

We can strengthen Theorem \ref{main-thm} a little further as follows. 
\begin{thm}\label{thm2}
Let $\mathcal{X}$ be a {\it compact metric space}. Assume that $\K(\infty)$ acts on $\X$ by isometric isomorphisms.  If $(x_n)_{n\in\N}$ is a convergent sequence in $\X$ and if $\mu \in\mathcal{P}(\X)$ is a weak limit point of the sequence of orbital measures $(\mu_n^{x_n})_{n\in\N}$, then $\mu$ is ergodic.  
\end{thm}

As a consequence, we have 
\begin{cor}
Let $\mathcal{X}$ be a {\it compact metric space}. Assume that $\K(\infty)$ acts on $\X$ by isometric isomorphisms. Then  a probability measure $\mu\in\mathcal{P}(\X)$ is ergodic if and only if  there exists a sequence of points $(x_n)_{n\in\N}$ in $\X$, such that $\mu$ is a weak limit point of a sequence of orbital measures $(\mu_n^{x_n})_{n\in\N}$. 
\end{cor}

\begin{rem}
In some important situations, without the assumption of compactness on the space $\X$ and without requiring that the group action is isometric, the same result in Theorem \ref{thm2} still holds, for instance, see \cite[Corollary 4.2]{OV-ams96}. 
\end{rem}

We also provide a simple example showing that, in general, the assumption that the action is isometric can not be dropped. 
\begin{prop}\label{prop-iso}
There exists an inductively compact group  $\K(\infty)$,  a compact metric space $\X$, on which the group $\K(\infty)$ acts by homeomorphisms, and a sequence $(x_n)_{n\in\N}$ of points in $\X$, such that non of the probability measures on $\X$ that are weak limits of $(\mu_n^{x_n})_{n\in\N}$ is ergodic.
\end{prop}

\section{Proof of Theorem \ref{main-thm}}
Our proof is based on the following Reverse Martingale type result from \cite{Bufetov-erg-dec}. For introducing this result, we need more notation. Given any bounded measurable function $f: \mathcal{X} \rightarrow \R$, we may define, for any positive integer $n\in\N$,  a function $A_n f$ by 
\begin{align}\label{av-op}
(A_n f)(x) : = \int_{\mathcal{X}} f  d\mu_n^x =   \int_{\mathcal{K}(n)}  f(g\cdot x) m_n (dg).
\end{align}
Let $\nu$ be any $\mathcal{K}(\infty)$-invariant probability measure.  By using the Reverse Martingale theorem, one may prove that the following limit 
\begin{align}\label{mart-limit}
(A^\nu_\infty f)(x) : = \lim_{n\to \infty}(A_n f)(x)
\end{align}
exists for $\nu$-almost every $x\in\mathcal{X}$. Moreover, $\nu$ is ergodic if and only if  there exists a dense subset $\Psi \subset L_1(\mathcal{X}, \mu)$ such that for any $\psi\in\Psi$, one has 
\begin{align}\label{erg-av}
\text{$(A^\nu_\infty \psi)(x)  = \int_{\mathcal{X}} \psi  d\nu$, for $\nu$-almost every $x\in\mathcal{X}$.}
\end{align}
See \cite[Propositions 6, 7 and 8]{Bufetov-erg-dec} for the details of the above statements.  

Let us now  proceed to the proof of Theorem \ref{main-thm}. By assumption, assume that there exists a subsequence $(n_i)_{i\in\N}$ of the natural numbers, such that 
$$
\text{$\mu_{n_i}^{x_0} \Longrightarrow \mu$ as $i \to \infty$.} 
$$
Note that if we replace the chain $\K(1) \subset \cdots \K(n) \subset\cdots$ by the chain $\K(n_1) \subset \cdots \K(n_i) \subset\cdots$, then both chains define a same inductive limit group $\K(\infty)$.  By this observation, without loss of generality, we may thus assume that 
$$
\text{$\mu_{n}^{x_0} \Longrightarrow \mu$ as $n \to \infty$.} 
$$
First note that if $x \in \K(\infty)\cdot x_0$, then 
\begin{align}\label{limit-m}
\text{$\mu_{n}^{x} \Longrightarrow \mu$ as $n \to \infty$.} 
\end{align}
Indeed, this follows from the elementary fact that once $x\in \K(N)\cdot x_0$, then for any $n\ge N$, we have $\mu^x_n = \mu_n^{x_0}$. 
By definition of weak convergence of probability measures, for any bounded continuous function $f: \X \rightarrow \R$, we have
\begin{align*}
\lim_{n\to\infty}\int_\X f (y)  \mu_n^{x}(dy)   = \int_\X f (y)  \mu(dy) \quad (x \in \K(\infty) \cdot x_0). 
\end{align*}
In the notation \eqref{av-op}, this means 
\begin{align}\label{to-extend}
\lim_{n\to\infty}  (A_n f)(x)   = \int_\X f   d \mu \quad (x \in \K(\infty) \cdot x_0). 
\end{align}
Since the support  $\supp(\mu_n^x)$ of the measure $\mu_n^x$ equals $\K(n)\cdot x$, we see that $ \supp(\mu_n^x)$  is a subset of the closure $\clos(\K(\infty)\cdot x_0)$ of the orbit $\K(\infty)\cdot x_0$. By \eqref{limit-m}, we get
\begin{align}\label{supp}
\supp(\mu) \subset \clos (\K(\infty)\cdot x_0). 
\end{align}

We need the following simple lemma. Let us denote the metric on $\X$ by $d_\X(\cdot, \cdot)$. 
\begin{lem}\label{lem-equicontinuity}
The  continuous functions in the sequence $(A_n f )_{n\in\N}$ are equicontinuous, that is, for any $\varepsilon>0$, there exists $\delta>0$, such that whenever $y, z \in\X$ satisfy $d_\X(y, z)\le \delta$, then 
$$
\sup_{n\in\N} | A_n f(y) - A_nf(z)| \le \varepsilon.
$$
\end{lem}

\begin{proof}
We will use the fact that the group action of $\K(\infty)$ on $\X$ is an isometric action. Since $f$ is a continuous function on the compact space $\X$, it is uniformly continuous. Thus for any $\varepsilon>0$, there exists $\delta>0$, such that  whenever $y, z \in\X$ satisfy $d_\X(y, z)\le \delta$, then $|  f(y) - f(z)| \le \varepsilon$. Thus under the above condition on $y, z$,  for any $g \in \K(\infty)$, we have $d_\X(g\cdot y, g \cdot z) = d_\X(y, z) \le \delta$ and hence $|  f(g \cdot y) - f( g \cdot z)| \le \varepsilon$. 
It follows that for any $n\in\N$, we have 
\begin{align*}
| A_n f(y) - A_n f(z)|   & \le \left| \int_{\mathcal{K}(n)}  f(g\cdot y) m_n (dg) -  \int_{\mathcal{K}(n)}  f(g\cdot z) m_n (dg)  \right| \le \varepsilon. 
\end{align*}
This proves the equicontinuity of the family of continuous functions $\{A_nf: n \in\N\}$.
\end{proof}

By using Arzel\`a-Ascoli theorem,  \eqref{to-extend}  can actually be strengthened to: 
\begin{align}\label{extend}
\lim_{n\to\infty}  (A_n f)(x)   = \int_\X f   d \mu \quad (x \in \clos( \K(\infty) \cdot x_0)). 
\end{align}
For the sake of completeness, let us give more details for the above assertion. First, if $x \in \clos( \K(\infty) \cdot x_0)$, then the sequence $((A_n f)(x))_{n\in\N}$ is a Cauchy sequence and hence converges. Indeed, for any $\varepsilon>0$, let $\delta>0$ be chosen as in Lemma \ref{lem-equicontinuity}.  We may find $y\in  \K(\infty) \cdot x_0$, such that  $d_{\X}(y, x)\le \delta$, hence $\sup_{n\in\N} | A_n f(y) - A_nf(x)| \le \varepsilon$. It follows that 
\begin{align*}
 | A_n f(x) - A_mf(x)|  \le &  | A_n f(x) - A_nf(y)|   +  | A_n f(y) - A_mf(y)| 
 \\
 &  +  | A_m f(y) - A_mf(x)| 
 \\
 \le & 2 \varepsilon  + | A_n f(y) - A_mf(y)|.
\end{align*}
Now we may use the fact derived from \eqref{to-extend} that  $(A_n f(y))_{n\in\N}$ is a Cauchy sequence to conclude that so is the sequence $((A_n f)(x))_{n\in\N}$.  Now we have proved that the sequence of functions $(A_n f|_{ \clos( \K(\infty) \cdot x_0)})$, all defined on a compact set $\clos( \K(\infty) \cdot x_0)$,  is uniformly bounded and equicontinuous and converges pointwisely, hence by Arzel\`a-Ascoli theorem, it converges uniformly on  $\clos( \K(\infty) \cdot x_0)$. This completes the proof of \eqref{extend}.

Finally, in view of \eqref{supp}, we have proved that 
\begin{align}\label{extend}
\lim_{n\to\infty}  (A_n f)(x)   = \int_\X f   d \mu \text{\, for $\mu$-almost every $x\in\X$}. 
\end{align}
Take $\Psi = C(\X)$, the set of continuous functions on $\X$. Since $\Psi$  is dense in $L^1(\X, \mu)$, we may apply characterization of ergodic measures \eqref{erg-av} to conclude that $\mu$ is ergodic. 

\section{Proof of Theorem \ref{thm2}}

By Theorem \ref{main-thm}, to prove Theorem \ref{thm2}, it suffices to show that if the sequence $(x_n)_{n\in\N}$ in $\X$ converges to a point $x_0\in\X$ such that $\mu_{n}^{x_n} \Longrightarrow \mu$ as $n \to \infty$, then $\mu_{n}^{x_0} \Longrightarrow \mu$ as $n \to \infty$. 

Let $f$ be any continuous function on $\X$. By assumption, we have 
\begin{align*}
\lim_{n\to\infty}  \int_\X f d \mu_n^{x_n}  = \int_\X f d \mu. 
\end{align*}
For any $\varepsilon>0$, let $\delta>0$ be chosen such that whenever $d_\X(x,y)\le \delta$, then $| f(x) - f(y)| \le \varepsilon$.  There exists  $n_0$,  whenever  $n \ge n_0$, we have $d_{\X}(x_n, x_0) \le \delta$.  Using again our assumption that the group action is an isometric action, for any $n\ge n_0$, the inequality $\left| f(g\cdot x_n)  - f(g\cdot x_0) \right| \le \varepsilon $ holds for any $g \in\K(\infty)$. Consequently, for any $n\ge n_0$, we have
\begin{align*}
\left|  \int_\X f d \mu_n^{x_n}  -  \int_\X f d \mu_n^{x_0} \right|  \le  \int_{\mathcal{K}(n)}   \left| f(g\cdot x_n)  - f(g\cdot x_0) \right|   m_n (dg ) \le \varepsilon.
\end{align*}
It follows that 
\begin{align*}
\limsup_{n\to\infty} \left|  \int_\X f d \mu_n^{x_n}  -  \int_\X f d \mu_n^{x_0} \right| = 0. 
\end{align*}
Hence the relation 
\begin{align*}
\lim_{n\to\infty}  \int_\X f d \mu_n^{x_0}  = \int_\X f d \mu
\end{align*}
holds for any continuous function $f$ on $\X$, that is, 
$$
\text{$\mu_{n}^{x_0} \Longrightarrow \mu$ as $n \to \infty$}.
$$
This completes the proof of Theorem \ref{thm2}. 

\section{Proof of Proposition \ref{prop-iso}}
Let $\K(n) = \Z/2^n\Z$. Note that we have the natural inclusion 
$$
\K(n)  = \Z/2^n\Z \simeq   2 \Z/2^{n+1}\Z  \subset  \Z/2^{n+1}\Z = \K(n+1).
$$
The corresponding group inductive limit $\K(\infty)$ is called the Pr\"ufer 2-group and is usually denoted by $\Z(2^\infty)$. Obviously, $\Z(2^\infty)$ is a countable group, which is inductively compact but not compact.   Let $\X = \{0, 1\}^{\K(\infty)}$, then $\X$ is a compact metrizable space.   For instance, we may fix the following metric on $\X$: if $x = (x(g))_{g\in\K(\infty)}$ and $y = (y(g))_{g\in\K(\infty)}$, then
$$
d_\X(x, y) = \sum_{n= 1}^\infty 4^{-n} | \{g\in\K(n):   x(g) \ne y(g) \}|, 
$$
where $| \{g\in\K(n):   x(g) \ne y(g) \}|$ denotes the  cardinality of  the finite set $\{g\in\K(n):   x(g) \ne y(g) \}$. The group $\K(\infty)$ acts naturally on $ \{0, 1\}^{\K(\infty)}$ by translations. This action is not an isometric action. 

Let us fix the natural inclusions $ \K(n)   \subset  \K(n+1)$. For any $n\in\N$, we have, 
\begin{align*}
\K(n+1) &=  \K(n)  \sqcup ( \K(n+1)  \setminus \K(n) ).
\end{align*}
Hence $\K(n+1)$ is the union of two $\K(n)$-cosets in $\K(n+1)$.   Now we set $x_{n} \in\X$ as follows:  $x_{n}: \K(\infty) \rightarrow\{0, 1\}$ is a function defined inductively by the following: on $\K(n)$, we set $x_{n}|_{\K(n)}= \1_{\K(n-1)}$; on $\K(n+1)$,   we extend the definition $x_{n}$ such that its restriction on the nontrivial $\K(n)$-coset  $g_n + \K(n)$ with $g_n \in \K(n+1) \setminus \K(n)$  is a translation of $x_n|_{\K(n)}$; if we have defined $x_n|_{\K(m)}$  for a positive integer $m\ge n$, then we may extend it to a function $x_n|_{\K(m+1)}$ such that the restriction of $x_n|_{\K(m+1)}$ on the nontrivial $\K(m)$-coset $g_m + \K(m)$ with $g_m \in \K(m+1) \setminus \K(m)$, is given by 
$$
x_n|_{g_m + \K(m)} (\cdot) = x_n|_{\K(m)} ( \cdot - g_m). 
$$
It is not hard to see that the orbit $\K(n)\cdot x_n$ consists of two points (both viewed as $\{0,1\}$-valued functions defined on $\K(\infty)$): 
$$
\K(n)\cdot x_n = \{x_n, 1 - x_n\}. 
$$
And we have 
$$
\mu_{n}^{x_n} = \frac{1}{2}(\delta_{x_n} + \delta_{1-x_n}), 
$$
where $\delta_{x_n}$ and $\delta_{1-x_n}$ are Dirac measures on $x_n$ and $1-x_n$ respectively. By construction, we have the following convergence in the space $\X$: 
$$
\lim_{n\to\infty}x_n  = \bar{1} \an \lim_{n\to\infty} (1-x_n) = \bar{0},
$$ 
where $\bar{0}, \bar{1}\in\X = \{0,1\}^{\K(\infty)}$ are constant functions on $\K(\infty)$ taking values $0$ and $1$ respecively. Consequently, 
$$
\text{$\mu_n^{x_n} \Longrightarrow   \frac{1}{2}(\delta_{\bar{1}} + \delta_{\bar{0}}) $ as $n\to \infty$,}
$$
 Since the singletons $\{\bar{0}\}$ and $\{\bar{1}\}$ are both  $\K(\infty)$-invariant,  the weak limit probability measure $\frac{1}{2}(\delta_{\bar{1}} + \delta_{\bar{0}})$ is not ergodic. This completes the proof of Proposition \ref{prop-iso}.


\begin{thebibliography}{1}

\bibitem{Bufetov-erg-dec}
A.~I. Bufetov.
\newblock Ergodic decomposition for measures quasi-invariant under {B}orel
  actions of inductively compact groups.
\newblock {\em Mat. Sb.}, 205(2):39--70, 2014.

\bibitem{OV-ams96}
Grigori Olshanski and Anatoli Vershik.
\newblock Ergodic unitarily invariant measures on the space of infinite
  {H}ermitian matrices.
\newblock In {\em Contemporary mathematical physics}, volume 175 of {\em Amer.
  Math. Soc. Transl. Ser. 2}, pages 137--175. Amer. Math. Soc., Providence, RI,
  1996.

\bibitem{Vershik-inf-group}
A.~M. Ver{\v{s}}ik.
\newblock A description of invariant measures for actions of certain
  infinite-dimensional groups.
\newblock {\em Dokl. Akad. Nauk SSSR}, 218:749--752, 1974.

\end{thebibliography}

\def\cprime{$'$} \def\cydot{\leavevmode\raise.4ex\hbox{.}}

\end{document}